\documentclass[11pt,A4paper]{article}

\usepackage[left=30mm,right=30mm,top=25mm,bottom=30mm]{geometry}
\usepackage{labelfig}
\usepackage{epsfig}
\usepackage{epstopdf}

\usepackage{xfrac}

\usepackage[percent]{overpic}

\usepackage{color}
\usepackage{amsthm,amsmath,amssymb}
\usepackage{booktabs}
\usepackage{mathpazo}
\usepackage{microtype}
\usepackage{overpic}
\usepackage{bm}
\usepackage{sectsty}
\usepackage[
	pdftitle={PDFTitle},
	pdfauthor={Hugo Parlier},
	ocgcolorlinks,
	linkcolor=linkred,
	citecolor=linkred,
	urlcolor=linkblue]
{hyperref}

\definecolor{linkred}{RGB}{0,191,255} 
\definecolor{linkblue}{RGB}{16, 78, 139}

\usepackage[hang,flushmargin]{footmisc}
\usepackage{enumitem}
\usepackage{titlesec}
	\titlespacing{\section}{0pt}{12pt}{0pt}
	\titlespacing{\subsection}{0pt}{6pt}{0pt}
	
\titlelabel{\thetitle.\quad}

\makeatletter 

\long\def\@footnotetext#1{%
\H@@footnotetext{%
\ifHy@nesting 
\hyper@@anchor{\@currentHref}{#1}%
\else 
\Hy@raisedlink{\hyper@@anchor{\@currentHref}{\relax}}#1%
\fi 
}}

\def\@footnotemark{%
\leavevmode 
\ifhmode\edef\@x@sf{\the\spacefactor}\nobreak\fi 
\H@refstepcounter{Hfootnote}%
\hyper@makecurrent{Hfootnote}%
\hyper@linkstart{link}{\@currentHref}%
\@makefnmark 
\hyper@linkend 
\ifhmode\spacefactor\@x@sf\fi 
\relax 
}%

\ifFN@multiplefootnote%
\renewcommand*\@footnotemark{%
\leavevmode 
\ifhmode 
\edef\@x@sf{\the\spacefactor}%
\FN@mf@check 
\nobreak 
\fi 
\H@refstepcounter{Hfootnote}%
\hyper@makecurrent{Hfootnote}%
\hyper@linkstart{link}{\@currentHref}%
\@makefnmark 
\hyper@linkend 
\ifFN@pp@towrite 
\FN@pp@writetemp 
\FN@pp@towritefalse 
\fi 
\FN@mf@prepare 
\ifhmode\spacefactor\@x@sf\fi 
\relax%
}%
\fi 

\makeatother 

\theoremstyle{plain}
\newtheorem{theorem}{Theorem}[section]
\newtheorem{proposition}[theorem]{Proposition}
\newtheorem{lemma}[theorem]{Lemma}
\newtheorem{corollary}[theorem]{Corollary}

\theoremstyle{definition}
\newtheorem{definition}[theorem]{Definition}

\newtheorem{remark}[theorem]{Remark}

\newcommand{\R}{{\mathbb R}}

\newcommand{\N}{{\mathbb N}}
\newcommand{\Q}{{\mathbb Q}}

\newcommand{\Z}{{\mathbb Z}}

\newcommand{\arc}{{\mathcal A}}

\newcommand{\FF}{\mathcal F}

\sectionfont{\large \bfseries}
\subsectionfont{\normalsize}

\long\def\symbolfootnote[#1]#2{\begingroup%
\def\thefootnote{\fnsymbol{footnote}}\footnote[#1]{#2}\endgroup}

\def\blfootnote{\xdef\@thefnmark{}\@footnotetext}

\begin{document}

{\Large \bfseries Flip graphs for infinite type surfaces}

{\large Ariadna Fossas and Hugo Parlier\symbolfootnote[1]{\normalsize Supported by the Luxembourg National Research Fund OPEN grant O19/13865598.
}
}

{\bf Abstract.} 
We associate to triangulations of infinite type surface a type of flip graph where simultaneous flips are allowed. Our main focus is on understanding exactly when two triangulations can be related by a sequence of flips. A consequence of our results is that flip graphs for infinite type surfaces have uncountably many connected components.
\vspace{0.5cm}

\section{Introduction} \label{s:introduction}

A variety of simplicial complexes have been used to study surfaces, their homeomorphisms and their geometric structures. For finite type surfaces, arc and curve type graphs have been very useful tools for studying the geometry of different moduli spaces. In particular, flip graphs give a way of measuring distance between triangulations but also provide a coarse model for mapping class groups. In this article, we adapt flip graphs to the setting of infinite type surfaces. As one might expect, the passage to infinite type surfaces requires a little bit of care. 

Our starting point will always be a surface $\Sigma$ obtained by pasting together an infinite collection of triangles, and then removing the vertices who then belong to the space of ends of $\Sigma$. We define a graph $\FF(\Sigma)$ whose vertices are these triangulations up to isotopy and whose edges come from flipping arcs that lie in quadrilaterals. More precisely, two triangulations are joined by an edge if they are related by any number (possibly infinite) of flips that can be performed simultaneously. This adaptation of the usual flip graph has already been studied in the finite type case \cite{BoseEtAl,Disarlo-Parlier18}, and importantly, for infinite type surfaces it allows one to measure distances between a larger set of triangulations. If one only allow single flips, two triangulations can only be related if they differ by finitely many arcs (see \cite{Fossas-Nguyen} for an example of such a flip graph). 

Figuring out which triangulations are related by a sequence of flips in our setting is exactly the main result of this paper:

\begin{theorem}\label{thm:A}
Let $\Sigma$ be an infinite type surface. Let $S$ and $T$ be triangulations of $\Sigma$. Then, $S$ and $T$ are in the same connected component of $\mathcal{F}(\Sigma)$ if and only if there exists $K \geq 0$ such that for every $\alpha$ arc of $S$ and every $\beta$ arc of $T$ the intersection numbers $i(\alpha,T)$ and $i(\beta,S)$ are bounded by $K$.    
\end{theorem}

The proof of the theorem involves putting together a number of preliminary results, one of them being Proposition \ref{prop:convexity}, which uses a technique from \cite{Disarlo-Parlier19,Parlier-Weber} which shows that triangulations that share a multiarc form a convex subset of $\mathcal{F}(\Sigma)$. The rest of proof is mainly combinatorial, and relies on a graph coloring argument, and in particular Brooks' theorem. 

Using Theorem \ref{thm:A}, it is straightforward to construct examples of triangulations that are not related by sequences of flip transformations, showing that $\FF(\Sigma)$ has multiple connected components. In fact we show:

\begin{corollary}\label{cor:B}
For any $\Sigma$ of infinite type, $\mathcal{F}(\Sigma)$ has uncountably many connected components.
\end{corollary}

This is analogous to what happens for hyperbolic structures for infinite type surfaces. The graph $\mathcal{F}(\Sigma)$ can be thought of as a combinatorial analogue of Teichm\"uller space which, classically, is the space of conformal structures up to quasi-conformal map and also has uncountably many connected components \cite{Basmajian2}.There is a metric point of view to what we do: if one replaces triangles by ideal hyperbolic triangles, and pastes them with $0$ shear, the resulting hyperbolic metric is well defined. A flip now changes the hyperbolic metric, but the two metrics are bi-Lipschitz equivalent. Although this analogy is not concretely used here in any way, it would be interesting to explore to what extent $\mathcal{F}(\Sigma)$ provides a combinatorial model to spaces that arise in the smooth setting.

In a more quantitative direction, our methods give the following upper and lower bounds on flip distance between two triangulations in terms of the maximal intersection between the individual arcs of one triangulation and the other triangulation. The following is the combination of corollaries \ref{cor:lower} and \ref{cor:upper}.

\begin{corollary}\label{cor:upperandlower}
Let $K\geq 1$ be a constant and $T$ and $T'$ be triangulations of a surface $\Sigma$.

If every arc $\alpha$ of $T$ and every arc $\beta$ of $T'$, the intersection numbers $i(\alpha,T')$ and $i(\beta,T)$ are bounded above by $K$, then $T$ and $T'$ are related by at most
$$
2 \,K^2 \cdot 3^K - K^2
$$
flips.

Conversely, if $T$ contains an arc $\alpha$ that satisfies $i(\alpha,T')\geq K$, then the flip distance between $T$ and $T'$ is at least
$$
\log_4\left(3(K+1) \right).
$$
\end{corollary}
Note that as far as we can tell, the above statement is also new for finite type surfaces. 

{\bf Organization.}

The article is organized as follows. In Section \ref{sec:setup}, after some definitions and notation, we prove some results that provide the groundwork for the proving the main theorem. In particular, we show the easier direction of Theorem \ref{thm:A}, the proof of which is given in Section \ref{sec:mainproof}. In the final section (Section \ref{sec:connect}), we discuss some corollaries of our results, namely about the connected components of flip graphs (Corollary \ref{cor:B}). 

{\bf Acknowledgments.}

We thank Stefan Wenger for an inspiring conversation about bi-Lipschitz equivalent planar metrics when we first started working on this project.

\section{Setup and preliminary results}\label{sec:setup}

Let $\Sigma$ be a connected orientable surface obtained by the following procedure. We begin with a countable collection of triangles and paste the sides of triangles in pairs to obtain a connected orientable surface $\bar{\Sigma}$ (see Figure \ref{fig:Pasting} for an illustration). The image of the vertices of the triangles under the pasting is a collection of points which we denote by $P$ and call the ideal vertex set. We now set $\Sigma = \bar{\Sigma} \setminus P$. Note that each element of $P$ belongs to an end of $\Sigma$, and although they don't belong to $\Sigma$, this ideal vertex set is implicit when we use the notation $\Sigma$. Although it will also be used just for the topological surface obtained by this procedure, $\Sigma$ is really a pair consisting in the surface and its ideal vertex set. 

\begin{figure}[h]
\begin{center}
\includegraphics[width=10cm]{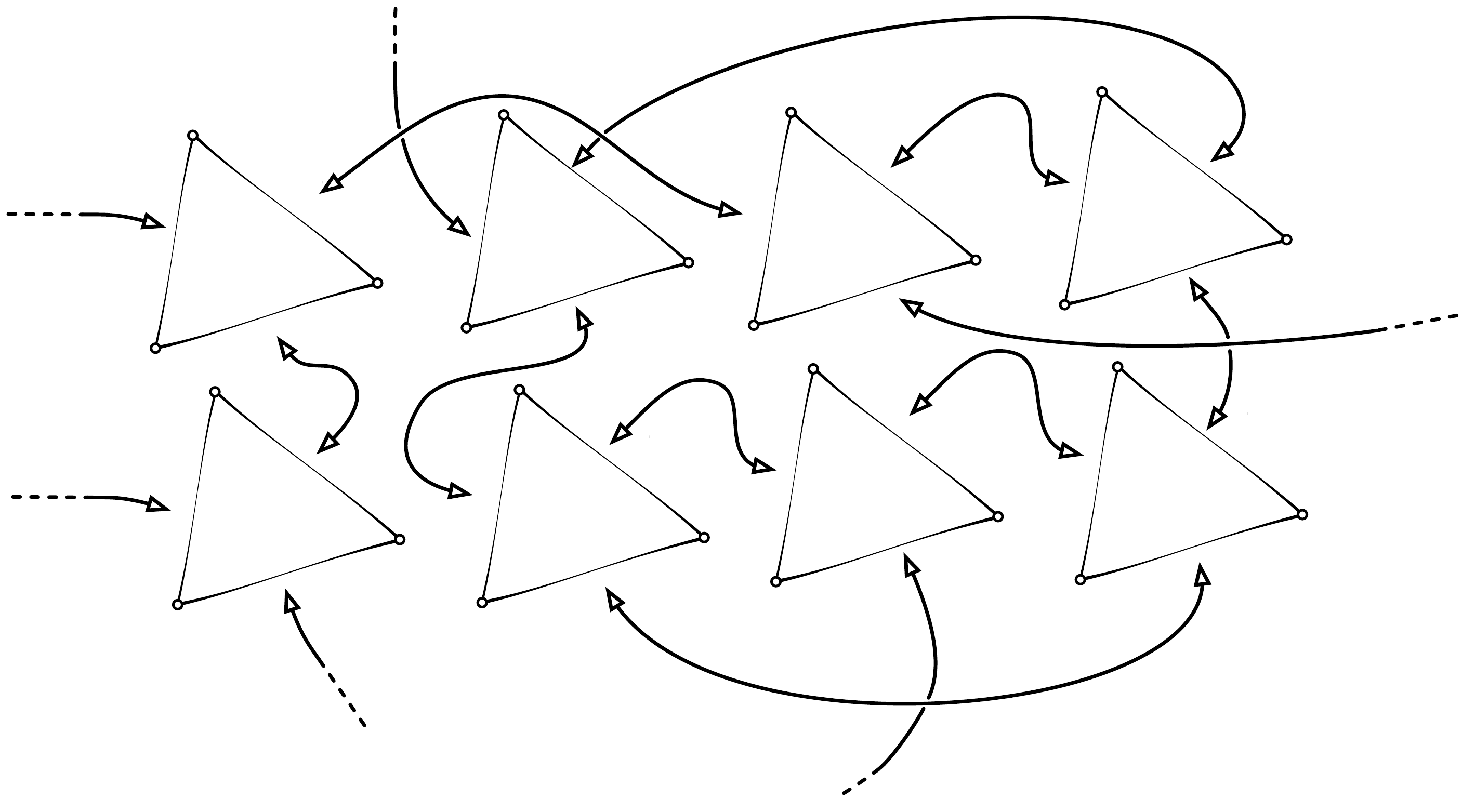}
\vspace{-24pt}
\end{center}
\caption{Pasting triangles to obtain $\Sigma$}
\label{fig:Pasting}
\end{figure}

Arcs of $\Sigma$ are non-trivial simple paths between (non-necessarily distinct) elements of $P$. We denote by $\arc(\Sigma)$ the set of arcs of $\Sigma$ up to isotopy fixing $P$ pointwise. We are interested in triangulations of $\Sigma$ by which we mean disjoint collections of arcs between elements of $P$ that decompose $\Sigma$ into a collection of (open) triangles. 

Note that arc and curve type graphs are becoming more understood in the context of infinite type surfaces \cite{Bavard,Durham-Fanoni-Vlamis,Fanoni-Ghaswala-McLeay, Aramayona-Fossas-Parlier} have been studied in different contexts and for different uses but, to the best of our knowledge, graphs with vertices being triangulations of infinite type surfaces have not yet been studied. 

\begin{remark}
An alternative approach would be to consider a surface with a fixed set of marked points and then consider maximal multiarcs (sets of isotopy classes of arcs, disjoint in their interior and maximal for inclusion). For finite type surfaces, these approaches are equivalent because given any surface (of negative Euler characteristic) and non-empty set of marked points, a maximal multiarc always decomposes the surface into triangles. For infinite type surfaces, if a triangulation exists then it is maximal but the converse is not always true. In fact, there are infinite type surfaces with a prescribed set of marked points $P$ which do not admit {\it any} triangulations with vertex set $P$. An easy example is the following: consider a one ended infinite genus surface and $P$ consisting of a single marked point. A maximal set of arcs contains an infinite number of arcs, and thus, if you trace a small circle around the marked point and look at the intersection with arcs (if necessary, realize the arcs with a metric) there will be accumulation points. These points correspond to  accumulation arcs, which cannot exist in a triangulation.
\end{remark}

Surfaces with sets of ideal vertices that can obviously be constructed via such a triangulation include $\R^2\setminus \Z^2$ with ideal vertex set $\Z^2$ as portrayed in figure \ref{fig:R2Z2}. This surface is commonly called the flute surface \cite{Basmajian}. 

\begin{figure}[h]
\begin{center}
\includegraphics[width=8cm]{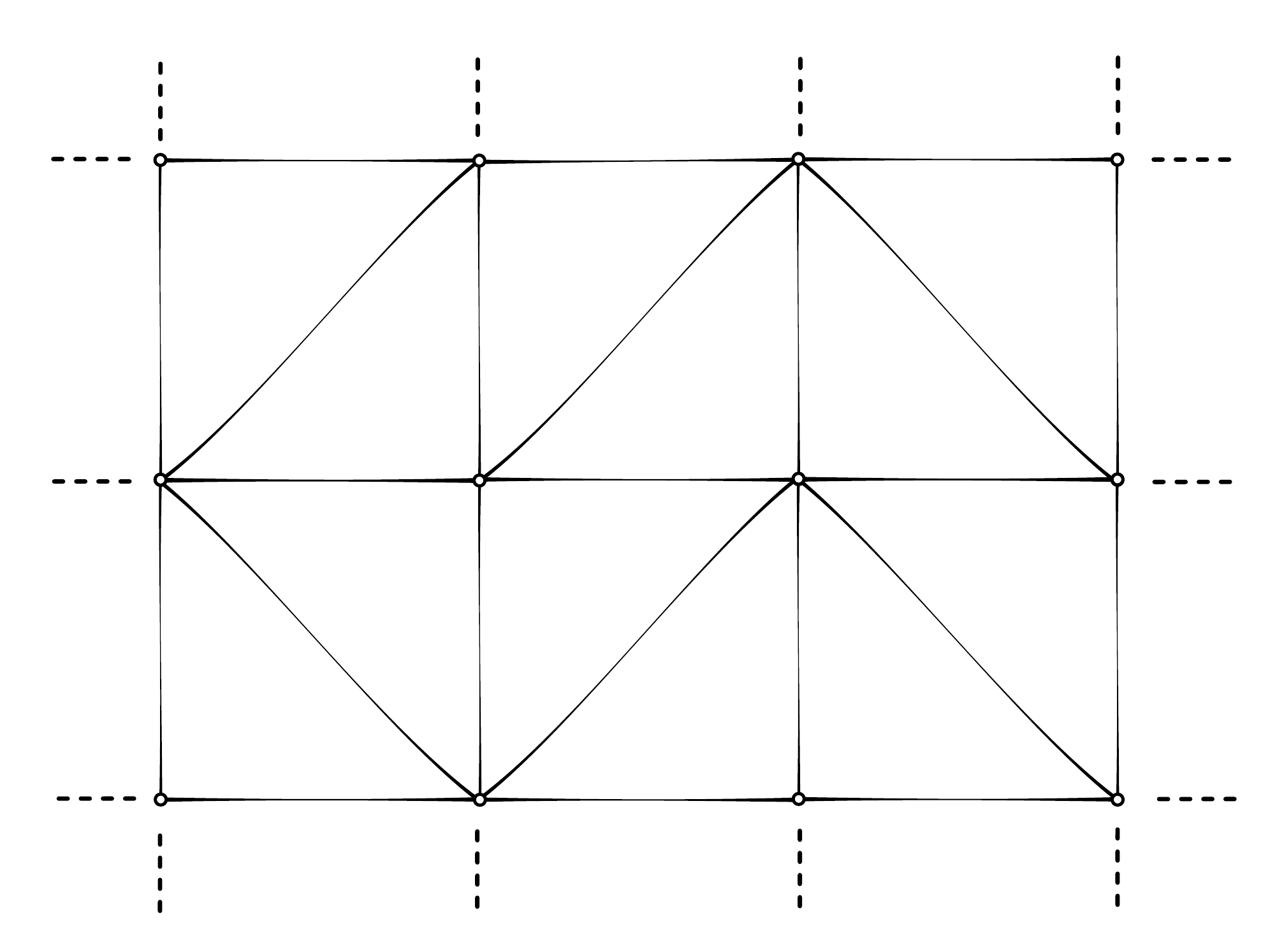}
\vspace{-24pt}
\end{center}
\caption{One way to (locally) paste triangles to obtain the flute surface}
\label{fig:R2Z2}
\end{figure}

A more subtle surface is obtained via the standard representation of the Farey graph in the hyperbolic plane with $\Q\cup \{\infty\}$ as ideal vertices. Note that in this case $\Sigma \cong \R^2$, so it only has one end, and all arcs leave and terminate in this end. And yet $P= \Q\cup \{\infty\}$, and there is an order on this set. Surfaces with infinite genus can also be obtained (for instance by adding genus to each triangle in the Farey graph tessellation to obtain the Loch Ness Monster surface as in Figure \ref{fig:Farey}), but, as seen previously, any arbitrary combination of $\Sigma$ and $P$ is not possible. 

\begin{figure}[h]
\begin{center}
\includegraphics[width=10cm]{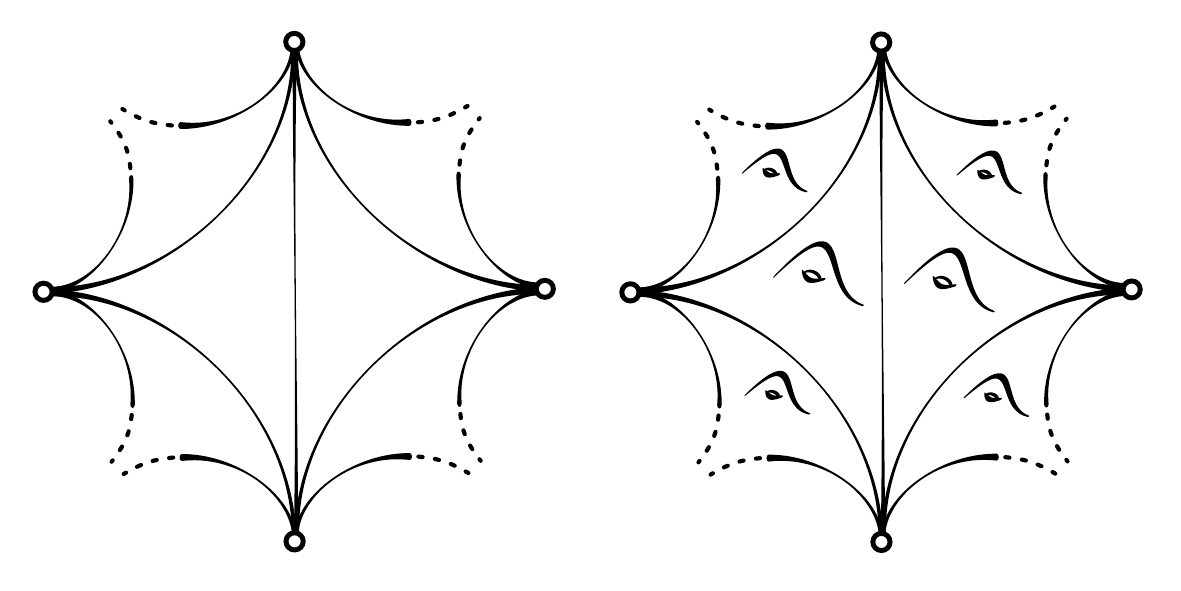}
\vspace{-24pt}
\end{center}
\caption{To the right of the Farey triangulation is the Loch Ness Monster: each "triangular" genus $1$ sub surface is easily triangulated}
\label{fig:Farey}
\end{figure}

Now given $\Sigma$, we define an associated flip graph $\mathcal{F}(\Sigma)$.

\begin{definition} Let $\mu$ be a (possibly infinite) multiarc of $T$ such that every arc $a\in \mu$ bounds two distinct triangles (that form a quadrilateral with $a$ as a diagonal). Suppose further that if $a,b \in \mu$ with $a\neq b$, the quadrilaterals containing $a$ and $b$ as diagonals are distinct. We define the triangulation $f_\mu(T)=T'$ to be the one obtained from $T$ by replacing every arc $a\in \mu$ by the other diagonal arc of the quadrilateral defined by the two triangles containing $a$. We say that $T$ and $T'$ are related by a \textit{simultaneous flip} or simply a {\it flip}.
\end{definition}

\begin{figure}[h]
\begin{center}
\includegraphics[width=8cm]{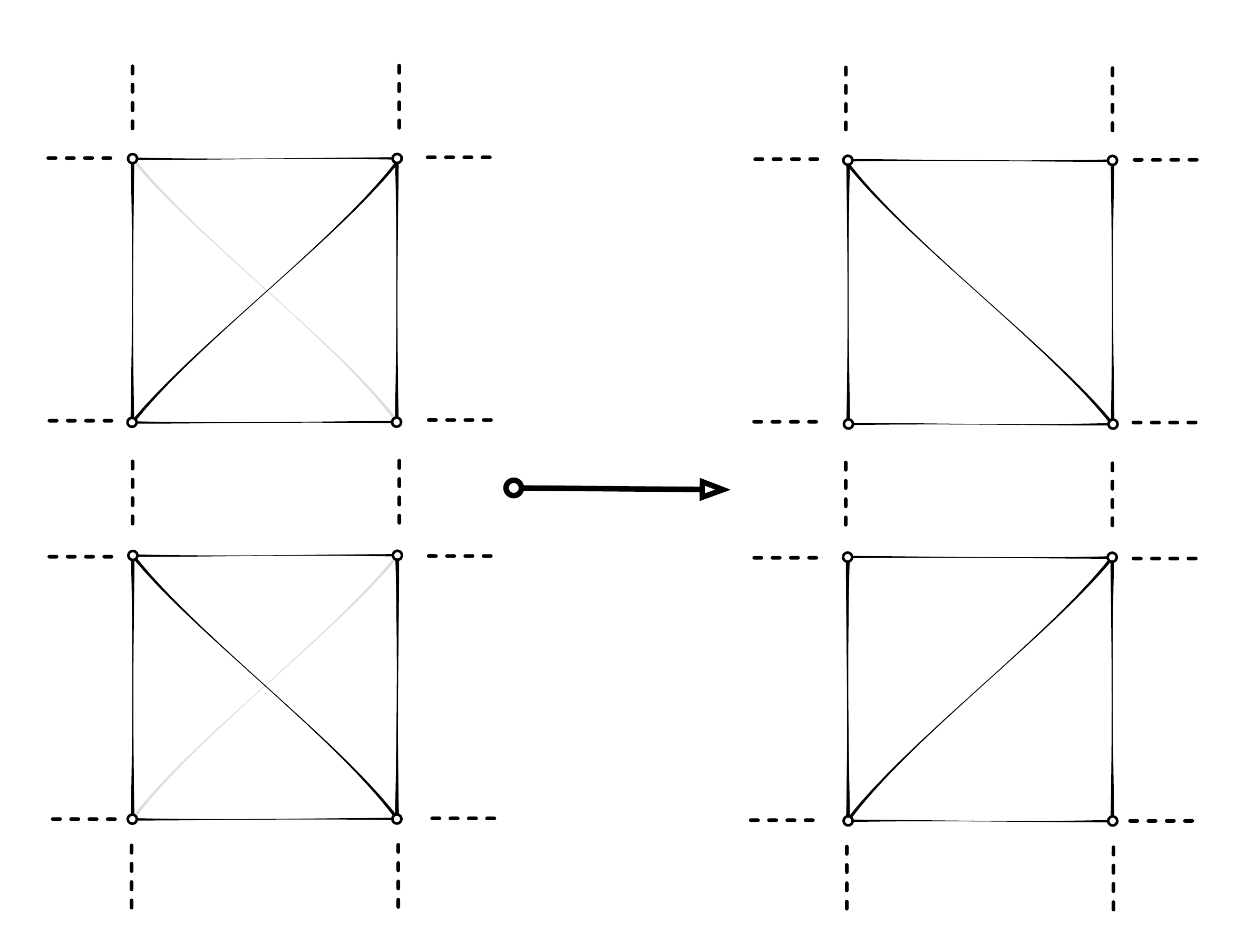}
\vspace{-24pt}
\end{center}
\caption{A simultaneous flip is done performing single flips simultaneously on disjoint quadrilaterals}
\label{fig:Flip}
\end{figure}

This allows us to define $\FF(\Sigma)$: vertices are the set of triangulations of $\Sigma$ and there is an edge between $T$ and $T'$ if they are related by a flip. We denote the connected components of $\FF(\Sigma)$ by $\mathcal{F}^i(\Sigma)$, $i\in I$, where $I$ is an index set. 

When $\Sigma$ is of finite type, $\mathcal{F}(\Sigma)$ is always connected \cite{Mosher} but we will see that when $\Sigma$ is of infinite type, it is always a graph with infinitely many connected components. 

The triangulations lying in a given connected component are formed of arcs of $\mathcal{A}(\Sigma)$. Our first observations are the following.

\begin{proposition}\label{prop:arcfinite}
Let $T_\alpha\in\mathcal{F}(\Sigma)$ and $\alpha\in \mathcal{A}(\Sigma)$. Then $i(\alpha,T) < +\infty$.
\end{proposition}
\begin{proof}
The argument is similar to above: if the arc intersected an infinite number of arcs of $T$, the intersection points would contain an accumulation point. This accumulation point cannot belong to the interior of a triangle, or to the interior of an edge of the triangle, so does not belong to $\Sigma$.
\end{proof}
\begin{proposition}\label{prop:arcexist}
Let $\mathcal{F}^i(\Sigma) \subset \mathcal{F}(\Sigma)$ be a connected component and $\alpha \in \mathcal{A}(\Sigma)$ an arc. Then there exists a triangulation $T_\alpha\in\mathcal{F}^i(\Sigma) $ containing $\alpha$. 
\end{proposition}
\begin{proof}
For any $T\in \mathcal{F}^i(\Sigma) $,  $i(T,\alpha) < +\infty$. Thus there is a finite type subsurface in which $\alpha$ and the subset of $T$ intersected by $\alpha$ both live. It suffices to flip in that subsurface to obtain a triangulation containing $\alpha$. 
\end{proof}

This implies that for any finite multiarc and any given connected component, there is a triangulation containing that multiarc. This is no longer true for infinite multiarcs. In particular, as mentioned previously, in the sequel there will be examples of triangulations that are not related by any finite number of simultaneous flips. 

The following proposition is the easy part of Theorem \ref{thm:A} and it states that if there is no bound on the intersection between arcs of a triangulation $T$ and an other triangulation $T'$, then they cannot be related by flips.

\begin{proposition}\label{prop:nocanflip}
Let $T,T' \in \FF(\Sigma)$ be such that for any $K>0$, there exists $\alpha \in T$ such that $i(\alpha,T') \geq K$, then $T$ and $T'$ cannot be related by a finite number of simultaneous flips.
\end{proposition}

\begin{proof}
For an integer $N>0$, consider a sequence of triangulations obtained from $T$ by flipping $N$ times. We denote the sequence by $T_0=T,\hdots,T_N$.

We set $K=\frac{1}{3}(4^{N+1}-1)$ for reasons that will become apparent in what follows, and let $\alpha$ be an arc with $i(\alpha,T')\geq K$. We will construct a sequence of arcs $\alpha_i\in T_i$ (by induction) for each $i=1,\hdots,N$, such that 
$$
i(\alpha_i,T')\geq \frac{1}{3}(4^{N+1-i}-1).
$$
This will show that $T_N\neq T'$, as $i(\alpha_N,T')>0$. 

We set $\alpha_0$ to be $\alpha$ to begin the induction.

At any stage, if the arc $\alpha_i$ does not belong to those that are flipped, we set $\alpha_{i+1}:=\alpha_i$ (and in particular its intersection number with $T'$ remains unchanged). If it does belong to an arc that is flipped, then $\alpha_i$ belongs to a quadrilateral $Q_i$ with four boundary arcs, the collection of which we denote $\partial Q_i$. Now the key observation is that the quantity 
$$
\max_{\delta\in \partial Q_i}\{i(\delta,T')\}
$$
is bounded below by a linear function of $i(\alpha_i,T')$. This is simply because any arc that intersects $\alpha_i$ must then intersect $\partial Q_i$ in both directions, unless it terminates at a vertex of $Q_i$. An arc that does not terminate in $Q_i$ contributes 2 to the intersection between $\partial Q_i$ and $T'$, and at arc that terminates contributes $1$, unless it is the other diagonal of $Q_i$, but there is only such diagonal. 

Hence $i(\partial Q_i, T') \geq i(\alpha_i,T') -1$, where the $-1$ is to account for the diagonal. Now as $Q_i$ has $4$ boundary arcs:
\begin{equation}\label{eq:induction}
 \max_{\delta\in \partial Q_i}\{i(\delta,T')\} \geq \frac{i(\alpha_i,T') -1}{4}
 \end{equation}
as claimed.

We now set $\alpha_{i+1}$ to be an arc of $Q_i$ which realizes this maximal intersection. 

We now finish by observing that 
$$
i(\alpha_{i+1}, T') \geq  \frac{i(\alpha_i,T') -1}{4} \geq \frac{\frac{1}{3}(4^{N+1-i}-1)-1}{4} =  \frac{1}{3}(4^{N-i}-1)
$$
as claimed. We point out that the choice of $K$ came from the inductive step (inequality \ref{eq:induction} above), and the fact that 
$$
\sum_{j=1}^{N} 4^j = \frac{1}{3} (4^{N+1-i}-1).
$$
So after any sequence of $N$ flips leaving from $T$, the resulting triangulation always has an arc that continues to intersect $T'$, and hence cannot be $T'$. As this is true for any $N$, the two triangulations are never connected by a sequence of flips. 
\end{proof}

The above proof also results in the following quantitative statement.

\begin{corollary}\label{cor:lower}
If $T$ and $T'$ are such that an arc $\alpha$ of $T$ satisfies 
$$
i(\alpha,T')\geq K
$$
for $K\geq 1$, then the flip distance between $T$ and $T'$ is at least
$$
\log_4\left(3(K+1) \right).
$$
\end{corollary}

\begin{proof}
The argument above shows that you need at least $N$ flips to relate $T$ and $T'$ if 
$$
i(\alpha,T')\geq   \frac{1}{3} (4^{N}-1).
$$
By setting $N=\log_4\left(3(K+1) \right)$, the result follows.
\end{proof}

One ingredient in our proofs will be the use of projections to natural subgraphs consisting in the triangulations that contain a given multiarc. These projections were at least implicitly studied in \cite{Mosher}, and were studied in detail in \cite{Disarlo-Parlier19}, but in both setups, $\Sigma$ was of finite type and the flip graphs involved were the "usual" ones, where only a single flip was allowed.

For a given multiarc $\mu$ of $\Sigma$ we define $\FF_\mu^{i}(\Sigma)$ to be the graph of $\FF^i(\Sigma)$ spanned by all vertices $T$ containing $\mu$. If $\mu$ is finite, then $\FF_\mu^{i}(\Sigma)$ is always non-empty but otherwise, as mentioned above, this might not be the case.

The following result has been shown in \cite{Parlier-Weber} for finite type surfaces, and is true in this more general setting as well. 

\begin{proposition}\label{prop:convexity}
$\FF_\mu^{i}(\Sigma)$ is a convex subgraph of $\FF^{i}(\Sigma)$.
\end{proposition}

\begin{proof}
Note the theorem is true in the finite type setting (for simultaneous flips this is proved in \cite{Parlier-Weber} using the strategy for individual flips from \cite{Disarlo-Parlier19}).

The basic observation that allows one to adapt the proof for the finite type setting here is the following. We suppose $\FF_\mu^{i}(\Sigma)$ is non-empty, that is that there exists at least one triangulation of $\Sigma$ which contains $\mu$. By Proposition \ref{prop:nocanflip}, this means that for any $T\in \FF^{i}(\Sigma)$ we have
$$
\sup_{\alpha \in \mu}i(\alpha,T) < +\infty.
$$
We are going to project a triangulation $T$ to a triangulation containing $\mu$ by "combing" $T$ along $\mu$ (see Figure \ref{fig:Comb}).

\begin{figure}[h]
\begin{center}
\includegraphics[width=10cm]{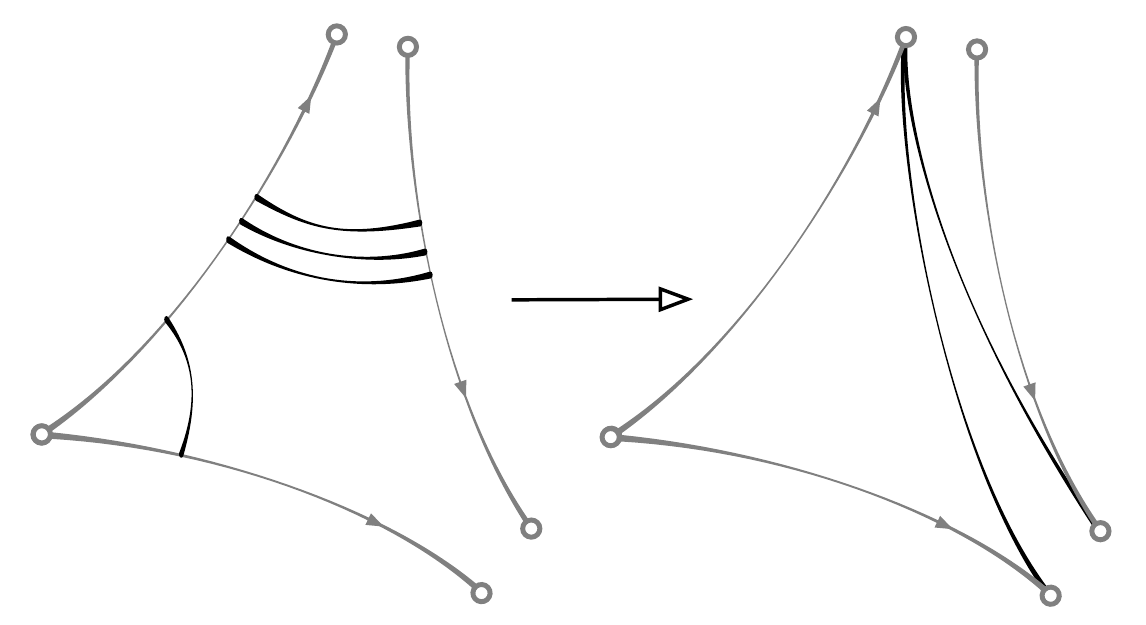}
\vspace{-24pt}
\end{center}
\caption{"Combing" a triangulation along an oriented multiarc}
\label{fig:Comb}
\end{figure}

The local picture will be the same as in the finite type setting as each arc of $\mu$ will intersect a bounded number of arcs of $T$. 

Consider a multiarc $\mu$, and give each of its arcs an orientation (the choice of orientation does not matter).

Now a combing projection works as follows. For a given $T\in \FF^{i}(\Sigma)$, each arc of $T$ is sent to a triangulation containing $\mu$ which is defined by "combing" $T$ along $\mu$ with the given orientation. Specifically one defines the map as follows: $T$ is intersected by $\mu$ and hence results in a collection of subarcs on which are the connected components of $\Sigma \setminus \mu$. For each such subarc $a$, we define an arc as follows: from an interior point of $a$ (that is a non terminal point), extend it in both directions until the terminal points of $a$. These are points of $\mu$ (possibly a marked point). If these are interior points of an arc $\alpha$ of $\mu$, continue the arc following $\alpha$ with its orientation to its terminal endpoint. This map clearly sends $a$ to an arc, and sends distinct subarcs $a,b$ to (interior) disjoint arcs. Note that different subarcs can be sent to the same arc.

One needs to check that the result is a triangulation, namely that the connected components of the complementary regions to the resulting multiarcs are all triangles. This is relatively straightforward to check: the rough argument is that otherwise there is a complementary region that is of greater (arc) complexity but which must have been intersected by arcs of $T$ or $\mu$. These arcs must have resulted in arcs that continue to intersect the region, hence this cannot be. A detailed argument in the finite type case can be found in both \cite{Disarlo-Parlier19} and \cite{Parlier-Weber}, and the same argument applies here.

Thus the combing map from $\FF^{i}(\Sigma)$ to $\FF_\mu^{i}(\Sigma)$ is well-defined on vertices. Now it suffices to show that edges are sent to edges. Roughly speaking, this can be deduced from the fact that quadrilaterals are sent to a collection of quadrilaterals, triangles or arcs, with at most one quadrilateral in the image. And if multiple quadrilaterals are disjoint, then their images are disjoint. Again, see \cite{Parlier-Weber} for a more detailed argument. 

Using this projection, convexity can be deduced from the observation that the combing map leaves points of $\FF_\mu^{i}(\Sigma)$ invariant. Given two triangulations $T,T'\in \FF_\mu^{i}(\Sigma)$, the image of geodesic under the combing map between them is a path between them of length at most $d(T,T')$ and entirely contained in $\FF_\mu^{i}(\Sigma)$.
\end{proof}

This proposition has an immediate consequence. Note that by a metric space of infinite rank we mean the absence of finite rank, meaning that for any positive integer $k$, it contains a quasi-convex copy $\Z^k$ with the usual metric. 

\begin{corollary}\label{cor:rank}
If $\Sigma$ is of infinite type, any connected component of the flip graph is of infinite diameter and infinite rank.
\end{corollary}
\begin{proof}
This follows from the fact that given any triangulation $T\in \FF^{i}(\Sigma)$, you can choose a multiarc $\eta\subset T$ that separates the surface into infinitely many finite type surfaces each with topology:
$$
\Sigma \setminus \eta = \dot\cup_{k\in \N} \Sigma_k.
$$
 We can now take a bi-infinite geodesic $\gamma_k$ on each of the flip-graphs of the finite type subsurfaces $\Sigma_k$. These are naturally collection of multiarcs that live on $\Sigma$ through the inclusion of $\Sigma_k$ in $\Sigma$. This gives rise to a collection of triangulations in $\FF^{i}(\Sigma)$ by choosing a point of $\gamma_k$ for all $k\in \N$ and taking the union with $\eta$. The convexity of $\FF_\mu^{i}(\Sigma)$ means that this is a convex subset giving us quasi-copies of $\Z^k$ for any $k\in \N$ (the "usual" metric and the "diagonal" metric are quasi-isometric).
\end{proof}

\section{The proof of Theorem \ref{thm:A}}\label{sec:mainproof}

This section is dedicated to proving Theorem \ref{thm:A}. We begin with a lemma, where we explicitly need a property of the combing projection from Proposition \ref{prop:convexity}.

\begin{lemma}\label{lem:intersect}
Let $T$ and $T'$ be triangulations of $\Sigma$ and let $\alpha$ be an oriented arc of $T$. Let $T'_{\alpha}$ be the combing projection of $T'$ onto $\mathcal{F}_{\alpha}(\Sigma)$. Then
$$\sup\{i(\beta',T) \, : \, \beta' \in T'_{\alpha}\} \leq \sup\{i(\beta,T) \, : \, \beta \in T'\}$$

and for every $\gamma \in T$, $i(\gamma,T'_{\alpha}) \leq i(\gamma,T')$. 

\end{lemma}
	
\begin{proof}
Note that the above quantities could be infinite. By definition, every arc $\beta'$ of $T'_{\alpha}$ is either $\alpha$, an arc of $T'$ or the concatenation of a subarc of a $\beta\in T'$ and a subarc of $\alpha$. In the first two cases, the lemma is obvious, and in the final case, the lemma follows as $i(\alpha,T)=0$ and thus $
i(\beta',T) \leq i(\beta,T).
$\end{proof}

\begin{lemma}\label{lem:coloring}
There exists a function $f:\N \to \N$ such that the following holds. Let $S$ and $T$ be two triangulations of $\Sigma$ satisfying
$$\max\{i(\alpha,T),i(\beta,S) \, : \, (\alpha,\beta) \in S \times T\} \leq K$$
for some finite $K$. Then there exist multiarcs $\mu_1, \ldots, \mu_{f(K)}$ of $S$ such that
\begin{enumerate}
\item $\displaystyle S= \bigcup_{i=1}^{f(K)} \mu_i$ and
\item if $i(\alpha,\gamma) > 0$ for $\alpha \in \mu_i$ and $\gamma \in T$, then $i(\alpha', \gamma)=0$ for all $\alpha' \in \mu_i\setminus\{\alpha\}$.
\end{enumerate}
The function can be taken to be 
$$
f(K)= 2 \cdot 3^K - 1.
$$

\end{lemma}
	
\begin{proof}
Consider the graph $G$ whose vertices are the arcs of $S$ and where two vertices are joined by an edge when the corresponding arcs belong to the boundary of the same triangle. Note that the degree of all vertices is uniformly bounded by 4. We can define a distance between two arcs of $S$ as the combinatorial distance of the corresponding vertices of $G$. 
		
Let $\alpha_1$ and $\alpha_2$ be two arcs of $S$ both intersected by some arc $\beta$ of $T$. The arc $\beta$ induces a path in $\Sigma$ from some point in the interior of $\alpha_1$ to some point in the interior of $\alpha_2$. Hence, this path crosses, by hypothesis, at most $K-2$ other arcs of $S$. Thus, the distance between $\alpha_1$ and $\alpha_2$ in $G$ is at most $K-1$.
		
Now take a multiarc $\mu$ of $S$ such that the distance in $G$ for every pair of different arcs of $\mu$ is at least $K$. This forces any arc of $T$ to intersect at most one arc in $\mu$.
		
Finally, consider the $K-th$ power of $G$ (the graph $G^K$ having the same vertex set as $G$ and where two vertices are joined by an edge if their distance in $G$ is at most $K$). The degree of all vertices of $G^K$ is uniformly bounded by $2\cdot 3^K -2$. By Brooks' coloring theorem \cite{Brooks} there exists a $(2\cdot 3^K -1)$-coloring of the vertices of $G^K$ such that any two adjacent vertices are of a different color. We can now define $\mu_1, \ldots, \mu_{f(K)}$ to be the monochromatic multiarcs and thus $f(K) \leq 2\cdot 3^K - 1$.   
\end{proof}

We can now proceed to the proof.

\begin{proof}[Proof of Theorem \ref{thm:A}.] 
		
Suppose first that $S$ and $T$ are two maximal triangulations in the same connected component of $\mathcal{F}(\Sigma)$ and consider a path $S=T_0, T_1, \ldots T_k=T$ joining them. By definition, $i(\alpha,T_0) \leq 1$ for all $\alpha \in T_1$. Note that if $\alpha$ belongs to $T_i$, every arc of $T_0$ crossing $\alpha$ either is the other diagonal of the quadrilateral of $T_i$ containing $\alpha$ or it crosses at least one of the boundary arcs of this quadrilateral. Thus, $i(\alpha, T_0) \leq 4 \max\{i(\beta,T_0) \, : \, \beta \in T_{i-1}\} + 1$. Hence, for all $\alpha \in T_k$ the intersection with $S$ is at most $\frac{4^k-1}{3}$. The same procedure applies to the path $T=S_0, T_{k-1}=S_1, \ldots, S_k=T_0=S$ showing that for all $\beta \in S$ the intersection with $S$ is also at most $\frac{4^k-1}{3}$.
		
On the other direction, suppose now that there exists $K \geq 0$ such that all arcs $\alpha$ of $S \cup T$ satisfy $i(\alpha, S \cup T) \leq K$. We now apply Lemma \ref{lem:coloring} to decompose the maximal multiarc $S$ into $\mu_1,\ldots,\mu_{f(K)}$. To do so we give every arc of $\mu_i$ an orientation. For $i \in \{1,\ldots, f(K)\}$ we define $T_i$ as the combing projection of $T_{i-1}$ along (oriented) $\mu_i$, with $T_0=T$. To see that such a simultaneous combing projection exists, let $\alpha, \alpha'$ be two different arcs of $\mu_i$ and denote $\Sigma_\alpha$ (respectively $\Sigma_{\alpha'}$) the subsurface of $\Sigma$ spanned by all triangles of $T_{i-1}$ intersecting $\alpha$ (respectively $\alpha'$). Then the interiors of $\Sigma_{\alpha}$ and $\Sigma_{\alpha'}$ are disjoint as a consequence of lemmas \ref{lem:intersect} and \ref{lem:coloring}. Furthermore, any arc of $T_{i-1}$ intersects at most one of $\alpha, \alpha'$. In addition the complexities of $\Sigma_\alpha$ are uniformly bounded because they can be triangulated with at most $K$ triangles. Hence, the total intersection number of $T_{i-1} \cap T_i \cap \Sigma_\alpha$ is bounded above by $K^2$ and the distance in the corresponding flip graphs is at most $K^2$ (see for instance Corollary 2.13 in \cite{Disarlo-Parlier19}). As these surfaces $\Sigma_\alpha$ have disjoint interiors, those flips can be done simultaneously, hence there is a path in $\mathcal{F}(\Sigma)$ between $T_{i-1}$ and $T_i$ of length at most $K^2$. This completes the proof.
\end{proof}

As a corollary to the above proof, we get the following quantitative statement (which does not require the surface to be of infinite type in any way). 

\begin{corollary}\label{cor:upper}
For a constant $K\geq 0$, let $S$ and $T$ be triangulations of a surface $\Sigma$ such that for every $\alpha$ arc of $S$ and every $\beta$ arc of $T$ the intersection numbers $i(\alpha,T)$ and $i(\beta,S)$ are bounded above by $K$. Then the simultaneous flip distance between $T$ and $S$ is bounded above by 
$$
2 \,K^2 \cdot 3^K - K^2.
$$
\end{corollary}

\begin{proof}
By Lemma \ref{lem:coloring} the function $f(K)$ satisfies 
$$
f(K)\leq 2 \cdot 3^K - 1
$$
and this is a bound on the number of "steps" necessary to get from $T$ to $S$. As argued above, each step requires at most $K^2$ (simultaneous) flips, hence the result. 
\end{proof}

\section{The topology of $\FF(\Sigma)$}\label{sec:connect}

We finish this paper with some observations about connected components of these flip graphs. Deformations spaces of infinite type surfaces generally have infinite numbers of connected components. For instance, the Teichm\"uller space of hyperbolic structures up to either quasi-conformal or bi-Lipschitz maps have this property (see \cite{Basmajian,Basmajian2,Liu-Papadopoulos}). One reason is because these maps only deform the lengths of geodesics by bounded amounts, and on infinite type surfaces, you have a lot of room to construct hyperbolic structures with wildly behaving length spectra \cite{Basmajian-Kim}.

In our setting, we use our result to show that our flip graphs have many connected components. 

\begin{corollary}\label{cor:notconnected}
For any $\Sigma$ of infinite type, $\mathcal{F}(\Sigma)$ has uncountably many connected components.
\end{corollary}

\begin{proof}
On $\Sigma$ there are infinitely many disjoint simple closed curves, say $\delta_i$, $i \in \N$. Now consider a triangulation $T\in \FF(\Sigma)$.

{\it Observation 1.} Each $\delta_i$ intersects finitely many arcs of $T$. 

We argue as is Proposition \ref{prop:arcfinite}. If not, there would be an accumulation of intersection points between the arcs and $\delta_i$. The accumulation point is a point of $\delta_i$, hence a point that belongs to either the interior or the interior of an edge of a triangle. But in either case, it cannot be the accumulation point of a collection of arcs.

{\it Observation 2.} Any arc of $T$ intersects a finite number of curves among $\delta_i$, $i\in \N$. 

This is because the curves are disjoint, so for instance can be completed into a pants decomposition and as our arcs are two ended, it cannot pass through infinitely many pairs of pants. 

So for each $\delta_i$, we have a finite collection of arcs of $T$ that intersect it, which in turn intersect a finite number of other curves. As a consequence we have:

{\it There exists a collection of pairs $\delta_{i_j}, \alpha_j$ where $j\in \N$, $\alpha_j\in T$, $i(\delta_{i_j}, \alpha_j) \neq 0$, and $i(\alpha_j, \delta_{i_{j'}})=0$ if $j\neq j'$.}

In other words, by successively choosing pairs of an intersecting curve and arc disjoint from all previous pairs, we get this infinite collection. Up to renumbering and relabelling, we can suppose that we have a collection $\delta_i,\alpha_i$, $i \in \N$, such that 
$$
i(\delta_i,\alpha_i) \leq i(\delta_{i+1},\alpha_{i+1}).
$$

Roughly speaking, we'll obtain new triangulations from $T$ by twisting the curves along increasing powers of Dehn twists. 

To obtain an uncountable number of connected components of $\FF(\Sigma)$, we begin by choosing a collection $\mathcal{C}$ of infinite subsets of $\N$ such that any two subsets contain infinitely many different elements. (Said in other way, the difference sets are always infinite.) 

Note that $\mathcal{C}$ is uncountable: it is of the same cardinality as infinite binary sequences which pairwise differ in infinitely many indices. And infinite binary sequences, without the difference condition, have the cardinality of $\R$. And for any given infinite binary sequence, there are at most a countable number of elements that only differ for a finite number of indices, hence our restricted set is also uncountable. 

Now we give each $\delta_i$ an orientation, and for $N\in \mathcal{C}$, we consider the homeomorphism $\varphi_N$ obtained by Dehn twisting $\delta_i$ $m_i$ times if $i\in N$ where $m_i$ is taken to be such that
$$
i(\varphi_N(\alpha_i), \alpha_i) \geq i.
$$
We now observe that if $M,N\in \mathcal{C}$ and $M\neq N$, then $\varphi_M(T)$ and $\varphi_N(T)$ belong to different connected components. Indeed, consider the sequence of integers $M\setminus N$, which is infinite by, and denote them $i_k$, $k\in N$. 

Thus the homeomorphism $\varphi_M$ acts on $\alpha_{i_k}$ by twisting $m_{i_k}$ times along $\delta_{i_k}$ but $\varphi_N$ leaves $\alpha_{i_k}$ invariant. Hence:
$$
i(\varphi_N(\alpha_{i_k}), \varphi_M(\alpha_{i_k}))= i(\alpha_{i_k}, \varphi_M(\alpha_{i_k})) \geq i_k. 
$$
And so there are arcs of $\varphi_M(T)$ with arbitrarily large intersection with arcs of $\varphi_N(T)$, and thus by Theorem \ref{thm:A}, lie in different connected components. 
\end{proof}

{\it Addresses:}\\
GAP Nonlinearity and Climate, Institut des Sciences de l'Environnement, Universit\'e de Gen\`eve, Suisse\\
Department of Mathematics, FSTM, University of Luxembourg, Esch-sur-Alzette, Luxembourg

{\it Emails:}\\
ariadna.fossastenas@unige.ch\\
hugo.parlier@uni.lu

\end{document}